\documentclass[12pt]{article}
\usepackage{amsmath}
\usepackage{graphicx,psfrag,epsf}
\usepackage{enumerate}
\usepackage{psfrag,epsf}
\usepackage{url} % not crucial - just used below for the URL
\usepackage{amsmath,amssymb,amsfonts,amsthm,mathtools,mathrsfs,booktabs}

\usepackage{tikz}
\usepackage{pgfplots}
\usepgfplotslibrary{fillbetween}
\pgfplotsset{compat=1.16}

\usepackage{bm,bbm}
\usepackage{color}
\usepackage{subfigure}
\usepackage{algorithm,algpseudocode}
\usepackage[symbol]{footmisc}
\usepackage{scalefnt}
\usepackage{authblk} % author and affiliations
\usepackage{multirow,centernot}
\usepackage[colorlinks=true, citecolor=blue, urlcolor=blue]{hyperref}
\usepackage{natbib}
\usepackage{dsfont}
\usepackage[title]{appendix}
\usepackage{newtxtext,newtxmath} % Times Roman fonts (text and math)
\usepackage{mhchem} % for \ce macro

\input cyracc.def

\usepackage[left=1in,top=1.1in,right=0.5in,bottom=1in]{geometry}

\theoremstyle{definition}

\newtheorem{theorem}{Theorem}[section]

\newtheorem{corollary}[theorem]{Corollary}

\newtheorem{proposition}[theorem]{Proposition}
\newtheorem{remark}[theorem]{Remark}
\makeatletter
\def\@seccntformat#1{\@ifundefined{#1@cntformat}%
	{\csname the#1\endcsname\quad}%      default
	{\csname #1@cntformat\endcsname}%    enable individual control
}
\makeatother

\newif\ifShowComments
\ShowCommentstrue
\def\strutdepth{\dp\strutbox}
\def\druk#1{\strut\vadjust{\kern-\strutdepth
        {\vtop to \strutdepth{%
                \baselineskip\strutdepth\vss
                        \llap{\hbox{#1}\quad}\null}}}}

\title{\bf
Bias analysis of a linear order-statistic 
 inequality \\
 index estimator: Unbiasedness under gamma populations
}

\author{
\text{Roberto Vila}$^{1}$\thanks{Corresponding author: Roberto Vila, email: {rovig161@gmail.com}
}
\,\, and
\text{Helton Saulo}$^{1,2}$ 
\\
{\small $^{1}$ Department of Statistics, University of Brasilia, Brasilia, Brazil}\\
{\small $^{2}$ Department of Economics, Federal University of Pelotas, Pelotas, Brazil}\\
}
\begin{document}
	\maketitle 	
	\begin{abstract}
This paper studies a class of rank-based inequality measures built from linear combinations of expected order statistics. The proposed framework unifies several well-known indices, including the classical Gini coefficient, the $m$th Gini index, the extended $m$th Gini index and particular cases of the $S$-Gini index, and also connects to spectral inequality measures through an integral representation. We investigate the finite-sample behavior of a natural U-statistic-type estimator that averages weighted order-statistic contrasts over all subsamples of fixed size and normalizes by the sample mean. A general bias decomposition is derived in terms of components that isolate the effect of random normalization on each rank level, yielding analytical expressions that can be evaluated under broad non-negative distributions via Laplace-transform methods. Under mild moment conditions, the estimator is shown to be asymptotically unbiased. Moreover, we prove exact unbiasedness under gamma populations for any sample size, extending earlier unbiasedness results for Gini-type estimators. A Monte Carlo study is performed to numerically check that the theoretical {unbiasedness} under gamma populations. {Finally, a data set on GDP per capita across $34$ countries in the Americas is analyzes to illustrate the proposed methodology.}
	\end{abstract}
%	
%%	\tableofcontents

\section{Introduction}

Quantifying economic inequality from sample data is a central problem in applied economics and statistics. The classical Gini coefficient is arguably the most widely used rank-based inequality measure, but it is well known to exhibit non-negligible finite-sample bias, especially for small and moderate sample sizes and for highly skewed populations; see, e.g., \cite{Deltas2003,Baydil2025} and references therein. This motivates the development and analysis of alternative indices and estimators with improved small-sample properties, while retaining the desirable invariance and rank-based interpretation that make Gini-type measures attractive in practice.

A prominent extension of the Gini coefficient is the class of generalized and extended Gini indices, including the $m$th Gini index, $S$-Gini index and its extensions, which can be expressed as normalized contrasts of expected order statistics \citep{Deltas2003,Gini1936,Gavilan-Ruiz2024,Vila2025,Vila2026,Yitzhaki-Schechtman2013}. Such measures are particularly appealing because they admit transparent interpretations in terms of dispersion between extreme or intermediate ranks and connect naturally to the broader class of spectral (rank-dependent) inequality measures \citep{Cowell2011,Chakravarty1990}. At the same time, many of these indices are typically estimated through U-statistic-type constructions that average rank contrasts over all subsamples of a given size $m$, often combined with normalization by the sample mean. Despite their conceptual appeal and practical relevance, the finite-sample bias behavior of these estimators has remained largely unexplored beyond specific cases.

The present paper contributes to this literature by developing a unified bias analysis for a broad family of order-statistic-based inequality indices and their canonical estimators. Specifically, for a non-negative random variable $X$ with finite, strictly positive mean $\mu=\mathbb{E}[X]$, and for $m\geqslant 2$, we consider the linear order-statistic inequality index
\begin{equation}\label{eq:rsrsf}
	I_m
	=
	\frac{1}{m\mu}\sum_{k=1}^m a_k\,\mathbb{E}[X_{k:m}],
	\quad
	\sum_{k=1}^m a_k=0,
	\quad a_1 \leqslant \cdots \leqslant a_m,
\end{equation}
which includes as special cases the classical Gini coefficient ($m=2$), the $m$th Gini index, the extended $m$th Gini index, and particular cases of the $S$-Gini index (Remark \ref{rem-1}). {The motivation of the proposed framework is mathematical. Note that all of the listed indices can be written as \eqref{eq:rsrsf}. This representation connects directly to the class of linear (spectral) inequality measures, which admit an integral representation $I_m = \mu^{-1}\int_0^1 w_m(u)Q_X(u)\,\mathrm{d}u$, where 	$
	w_m(u)=(1/m)\sum_{k=1}^m a_k f_{U_{k:m}}(u)
	$, \(Q_X\) is the quantile function of \(X\), and \(U_{k:m}\sim\mathrm{Beta}(k,m-k+1)\) is the \(k\)-th order statistic from a size-\(m\) sample of \(U(0,1)\), and are axiomatically characterized by Lorenz-consistency and the Pigou--Dalton transfer principle \citep{Cowell2011,Chakravarty1990}. The $S$-Gini, despite being derived from social welfare functions, belongs to this spectral class for specific polynomial weight functions, as demonstrated by \citet{Yitzhaki-Schechtman2013}. The economic interpretation is that all members of the family summarize rank-based dispersion in a single number that is invariant to scale, non-negative, and zero only under perfect equality.}

Given an independent and identically distributed (i.i.d.) sample $X_1,\ldots,X_n$ with $n\geqslant m$, we study the estimator $\widehat I_m$ defined in \eqref{estimator}, which averages the weighted order-statistic contrasts over all $m$-subsamples and normalizes by $\overline X$. Our first main result derives a general expression for the finite-sample bias $\mathrm{Bias}(\widehat I_m,I_m)$ (Corollary~\ref{def-bias}). The bias can be decomposed into a linear combination of factors
\[
\Delta_{n,r}
=
\mathbb{E}\!\left[\frac{X_{r:r}}{\overline X}\right]
-
\frac{\mathbb{E}[X_{r:r}]}{\mu},
\quad r=1,\ldots,m,
\]
which separates the contribution of each rank level in a manner that is amenable to analytic and numerical investigation. We further provide a characterization of $\Delta_{n,r}$ via Laplace-transform methods (Proposition \ref{prop-main}), yielding a route to compute (or approximate) the bias under general non-negative distributions.

Our second main result establishes asymptotic unbiasedness: under mild 
{moment and lower-tail} conditions, $\mathrm{Bias}(\widehat I_m,I_m)\to 0$ as $n\to\infty$ (Proposition~\ref{Asymptotic unbiasedness}). While asymptotic unbiasedness is reassuring, it does not address the small-sample bias that often drives empirical discrepancies. The key contribution of the paper is therefore our third main result, which identifies a distributional setting where the estimator is exactly unbiased for all sample sizes. Specifically, we prove that if $X\sim\mathrm{Gamma}(\alpha,\lambda)$, then $\Delta_{n,r}=0$ for all $r\leqslant n$, implying (Corollary~\ref{def-bias})
\begin{align*}
	\mathrm{Bias}(\widehat I_m,I_m)=0
	\quad\text{for all } n\geqslant m,
\end{align*}
thereby extending earlier unbiasedness results for the Gini coefficient and its $m$th and extended variants \citep{Deltas2003,Baydil2025,Vila2025,Vila2026}. The proof relies on the Dirichlet property of normalized gamma samples and the homogeneity of the maximum functional, which together yield the identity
$\mathbb{E}[X_{r:r}] = \mu\,\mathbb{E}[X_{r:r}/\overline X]$.

{The rest of the paper is organized as follows. Section~\ref{sec:02} defines the index $I_m$ and establishes its key properties and spectral integral representation. Section~\ref{Basic properties} presents several properties of the linear order-statistic inequality index. Section~\ref{general_bias} derives the general bias formula for $\widehat{I}_m$ in terms of the $\Delta_{n,r}$ factors and provides their Laplace-transform characterization. Section~\ref{Asymptotic unbiasedness} proves asymptotic unbiasedness under moment and lower-tail conditions. Section~\ref{Unbiasedness_gamma} establishes exact unbiasedness under gamma populations. Section~\ref{simulation study} presents Monte Carlo evidence comparing the proposed estimator with a plug-in parametric maximum likelihood estimator (MLE)  across gamma and non-gamma populations, together with a numerical validation of the bias decomposition via the $\Delta_{n,r}$ components. Section~\ref{empirical application} illustrates an empirical application to GDP per capita data in the Americas. Finally, Section~\ref{concluding remarks} provides some concluding remarks.}

\section{A linear order-statistic inequality index}\label{sec:02}

Let $X$ be a non-negative random variable with finite mean 
$\mu=\mathbb{E}[X]>0$. 
For an i.i.d. random sample of size $m\geqslant 2$, denote the order statistics by
$
X_{1:m} \leqslant \cdots \leqslant X_{m:m}.
$

Define the population index
\begin{align}\label{def-index}
	I_m
	\equiv
	I_m(X)
	=
	{1\over m\mu}
	{\displaystyle \sum_{k=1}^m a_k \, \mathbb{E}[X_{k:m}]},
\end{align}
where $(a_1,\dots,a_m)$ are fixed real coefficients such that
\begin{align}\label{conditions-an}
	\begin{aligned}
		\sum_{k=1}^m a_k = 0,
		\quad 
		a_1 \leqslant \cdots \leqslant a_m.
	\end{aligned}
\end{align}

The index $I_m$ measures a weighted contrast of expected rank positions relative to the population mean. 

If the distribution of $X$ is close to equality (i.e., $X_{k:m}\approx \mu$ for all $k$), then the index reflects low inequality (depending on the structure of the weights $a_k$). Conversely, large gaps between lower and upper order statistics increase the magnitude of the index, capturing stronger dispersion across ranks.

Therefore, $I_m$ can be viewed as a finite-dimensional, rank-based summary of inequality determined by the weighting scheme $(a_1,\dots,a_m)$.

Table \ref{Table-1} summarizes some main special cases of the linear order-statistic inequality index \eqref{def-index}.
\begin{table}[h!]
	\centering
	\begin{tabular}{p{3.8cm} p{4.1cm} p{3.0cm} p{5.0cm}}
		\toprule
		\textbf{Index} 
		& \textbf{Weights $a_k$} 
		& \textbf{Formula} 
		& \textbf{Interpretation} \\
		\midrule
		
		Classical Gini index $G$ \citep{Deltas2003,Gini1936}
		($m=2$) 
		& $a_1=-1,\; a_2=1$ 
		& $\displaystyle 
		%		\frac{\mathbb{E}[X_{2:2}-X_{1:2}]}{2\mu}
		%		=
		\frac{\mathbb{E}|X_1-X_2|}{2\mu}$ 
		& Mean absolute pairwise difference. \\[0.2cm]
		
		$m$th Gini index $IG_m$ 	
		\citep{Gavilan-Ruiz2024,Vila2025}
		($m\geqslant 2$) 
		& $a_1=-1,\; a_m=1,$ 
		
		$a_k=0$ for $2\leqslant k\leqslant m-1$ 
		& $\displaystyle 
		\frac{\mathbb{E}[X_{m:m}-X_{1:m}]}{m\mu}$ 
		& Dispersion between sample minimum and maximum. \\[0.2cm]
		
		Extended $m$th Gini 
		${IG}_m(j,k)$ \citep{Vila2026} 		
		($m\geqslant 2$) 
		& $a_j=-1,\; a_k=1,$ 
		
		$a_i=0$ for $i\neq j,k$ ,
		
		$1\leqslant j<k\leqslant m$
		& $\displaystyle 
		\frac{\mathbb{E}[X_{k:m}-X_{j:m}]}{m\mu}$ 
		& Difference between two arbitrary order statistics. \\[0.2cm]
		
		General linear 
		index $I_m$ 
		($m\geqslant 2$) 
		& Weights with 
		
		$\sum_{k=1}^m a_k=0$ and 
		
		$a_1 \leqslant \cdots \leqslant a_m$
		& $\displaystyle 
		\frac{1}{m\mu}
		\sum_{k=1}^m a_k\,\mathbb{E}[X_{k:m}]$
		& Most general linear order-statistic inequality index. \\
		%[0.3cm]
		%		
		%		 $R_\nu$ 
		%		
		%		($\nu>1$)  \cite{Yitzhaki-Schechtman2013}
		%		& $a_k=
		%		\nu
		%		\left[
		%		1
		%		-
		%		\frac{\binom{m-k+\nu-1}{\,m-k\,}}
		%		{\binom{m+\nu-1}{\,m\,}}
		%		\right],
		%		$
		%		
		%		
		%		& $1-\frac{\nu}{\mu}\,
		%		\mathbb{E}\!\left[X(1-F(X))^{\nu-1}\right]$ 
		%		& Gini with tunable tail sensitivity. \\
		
		\bottomrule
	\end{tabular}
	\caption{Special cases of the linear order-statistic index $I_m$.}
	\label{Table-1}
\end{table}

\begin{remark}\label{rem-1}
	Condition~\eqref{conditions-an} is required only for the functional $I_m$
	to be interpreted as an inequality index. It is not necessary for the validity
	of the main results and is used exclusively in Section~\ref{Basic properties}
	to derive lower and upper bounds for $I_m$. When this restriction is relaxed,
	several classical inequality measures are recovered as special cases of the
	general index $I_m$.
	
	In particular, by choosing
	\[
	a_k=\frac{1}{m}-b_k,\quad 
	b_1=1,\quad b_k=0\ \text{for }2\leqslant k\leqslant m,
	\]
	and
	\[
	a_k=b_k-\frac{1}{m},\quad 
	b_m=1,\quad b_k=0\ \text{for }1\leqslant k\leqslant m-1,
	\]
	we obtain, respectively, the extended lower and upper Gini indices introduced
	by \cite{Vila2025a}:
	\begin{align*}
		{}_iIG_{m}
		&=\frac{\mathbb{E}[X_{i}-X_{1:m}]}{m\mu},
		\quad 
		{}^{i}IG_{m}
		=\frac{\mathbb{E}[X_{m:m}-X_{i}]}{m\mu},
		\quad
		1\leqslant i\leqslant m.
	\end{align*}
	
	Moreover, by setting $a_1=1-m$, $m=\nu$, and $a_k=1$ for $k\geqslant 2$,
	the index $I_m$ yields particular cases of the $S$-Gini index $R_\nu$ $(\nu>1)$ (Gini coefficient with tunable tail sensitivity)
	\citep{Yitzhaki-Schechtman2013}.
\end{remark}

\section{Basic properties}\label{Basic properties}

In this section, we present several properties of the linear order-statistic inequality index defined in \eqref{def-index}.
\begin{itemize}
	\item {\bf Ratio-scale} invariance.
	For any $c>0$,
	$
	I_m(cX)=I_m(X).
	$
	
	\item {\bf Lack of translation invariance.}
	For any $c>0$, 
	$
	I_m(X+c)={\mu\over\mu+c} \, I_m(X).
	$
	
	\item {\bf Vanishing under equality.}
	If $X=c$ {(with $c>0$)} almost surely, then
	$
	I_m(X)=0.
	$
	
	\item {\bf Non-negativity.} Since
	\[
	\mathbb{E}[X_{1:m}] \leqslant \cdots \leqslant \mathbb{E}[X_{m:m}],
	\quad 
	a_1 \leqslant \cdots \leqslant a_m,
	\]
	Chebyshev's rearrangement inequality yields
	\[
	\frac{\displaystyle \sum_{k=1}^m a_k \, \mathbb{E}[X_{k:m}]}{m}
	\;\geqslant\;
	\frac{\displaystyle \sum_{k=1}^m a_k}{m}\,
	\frac{\displaystyle \sum_{k=1}^m \mathbb{E}[X_{k:m}]}{m}
	= 0,
	\]
	because $\sum_{k=1}^m \mathbb{E}[X_{k:m}]=m\mu$ and $\sum_{k=1}^m a_k=0$.
	Hence, the index satisfies $I_m \geqslant 0$.
	
	\item {\bf Upper bound.} Since $a_1 \leqslant \cdots \leqslant a_m$, we have
	\begin{align*}
		\frac{\displaystyle \sum_{k=1}^m a_k \, \mathbb{E}[X_{k:m}]}
		{m}
		\leqslant
		\frac{\displaystyle a_m\sum_{k=1}^m \, \mathbb{E}[X_{k:m}]}
		{m}
		=
		a_m\mu,
	\end{align*}
	because $\sum_{k=1}^m \mathbb{E}[X_{k:m}]=m\mu$. Therefore,
	$I_m\leqslant a_m$.
	
	\item {\bf Integral representation.}
	Using the representation
	\[
	\mathbb{E}[X_{k:m}]
	=
	\int_0^1 Q_X(u) f_{U_{k:m}}(u) {\rm d}u,
	\]
	where $Q_X$ is the quantile function of $X$ and $U_{k:m}\sim\text{Beta}(k,m-k+1)$ is the 
	$k$-th order statistic of a sample of size $m$ from $U(0,1)$, we obtain
	\[
	I_m
	=
	\frac{1}{\mu}
	\int_0^1 w_m(u)\,Q_X(u) {\rm d}u,
	\]
	with weight function
	$
	w_m(u)=(1/m)\sum_{k=1}^m a_k f_{U_{k:m}}(u).
	$
	
	Thus $I_m$ belongs to the class of linear (spectral) inequality
	measures \citep{Cowell2011,Chakravarty1990}.
	
	\item {\bf Covariance representation.}
	Let $F$ denote the cumulative distribution  function (CDF) of $X$.
	Since 
	\begin{align}\label{id-bin}
		\binom{m-1}{k-1}\mathbb{E}[F^{k-1}(X)[1-F(X)]^{m-k}]={1\over m}
		\quad 
		\text{and}
		\quad 
		\binom{m-1}{k-1}={k\over m} \,\binom{m}{k},
	\end{align}
	we have
	\begin{multline}\label{def-cov}
		{1\over\mu} 
		\sum_{k=1}^{m}a_k \,
		{\rm Cov}\left(X,\binom{m-1}{k-1}F^{k-1}(X)[1-F(X)]^{m-k}\right)
		\\[0,2cm]
		=
		{1\over m \mu} 
		\sum_{k=1}^{m}a_k
		\mathbb{E}\left[
		Q_X(U)
		\left\{
		k
		\binom{m}{k}U^{k-1}(1-U)^{m-k}
		\right\}
		\right],
	\end{multline}
	where $U\sim U(0,1)$ and $Q_X$ is the quantile function of $X$.
	By combining the well-known identity \cite[see Item (1) in][]{Kleiber2002}:
	\begin{align*}
		\mathbb{E}[X_{k:m}]
		=
		k
		\binom{m}{k}
		\mathbb{E}\left[Q_X(U)U^{k-1}(1-U)^{m-k}\right], 
	\end{align*}
	with the definition in \eqref{def-index} of $I_m$, from \eqref{def-cov} we have 
	\begin{align}\label{prop-using}
		I_m
		=	
		{1\over\mu} 
		\sum_{k=1}^{m}a_k \,
		{\rm Cov}\left(X,\binom{m-1}{k-1}F^{k-1}(X)[1-F(X)]^{m-k}\right).
	\end{align}
	
	\item {\bf Lorenz-type inequality representation.} The Lorenz curve $L$ for $X$ is defined by 
	$
	L(p)=(1/\mu){\int_0^p Q_X(t){\rm d}t},  \text{for any} \ 0\leqslant p\leqslant 1.
	$
	Since $L'(p)=Q_X(p)/\mu$, by \eqref{prop-using} and identity in \eqref{def-cov}, we have
	\begin{align*}
		I_m
		=
		{1\over m} 
		\sum_{k=1}^{m} a_k \,
		\mathbb{E}\left[
		\{L'(U)-1\}
		\left\{
		k
		\binom{m}{k}U^{k-1}(1-U)^{m-k}
		\right\}
		\right],
	\end{align*}
	where $U\sim U(0,1)$. By integration by parts, the above identity takes the form:
	\begin{align*}
		I_m
		=
		{1\over m} 
		\sum_{k=1}^{m} a_k \,
		\mathbb{E}\Bigg[
		\{U-L(U)\}
		\Bigg\{
		k\binom{m}{k}U^{k-2}(1-U)^{m-k-1}[k-1-(m-1)U]
		\Bigg\}
		\Bigg].
	\end{align*}
	Using the identity \eqref{id-bin} and the binomial theorem applied to $(1-U)^{m-k-1}$, we derive
	\begin{align*}
		I_m
		=
		\sum_{k=1}^{m} a_k \,
		\binom{m-1}{k-1}
		\sum_{r=0}^{m-k-1}
		\binom{m-k-1}{r}
		(-1)^{m-k-1-r} 
		\left[
		{k-1\over m-1-r} \,
		D_{m-2-r}
		-
		{m-1\over m-r} \,
		D_{m-1-r}
		\right],
	\end{align*}
	where 
	$
	D_n\equiv D_n(X)=(n+1)\mathbb{E}[\{U-L(U)\}U^{n-1}], 
	\ U\sim U(0,1), \ n\geqslant 1,
	$
	is the Lorenz measure of inequality for $X$ introduced in \cite{Aaberge2000}.

	\item {\bf Structural relationship.}
	$
	\text{Gini (}m=2\text{)}
	\subset
	\text{$m$th Gini}
	\subset
	\text{Extended $m$th Gini}
	\subset
	I_m
	\subset
	\text{Spectral inequality measures}.
	$
\end{itemize}

\begin{proposition}\label{ext-gini-index-0}
	For any  $m\geqslant 2$, the linear order-statistic inequality index \eqref{def-index} can be computed as
	\begin{align*}
		I_m
		=
		{1\over m} 
		\displaystyle
		\sum_{k=1}^m a_k 
		\sum_{r=k}^{m}
		\binom{m}{r}
		(-1)^{r-k}\binom{r-1}{k-1}
		\dfrac{ \displaystyle
			\mathbb{E}[X_{r:r}]
		}
		{\displaystyle
			\mu
		}.
	\end{align*}
\end{proposition}
\begin{proof}
	From Proposition 10 of \cite{Salama-Koch2020}, it follows that
	\begin{align*}
		X_{k:m}
		=
		\sum_{r=k}^{m}
		(-1)^{r-k}
		\binom{r-1}{k-1}
		\sum_{1\leqslant t_1<\cdots<t_r\leqslant m}
		\max\{X_{t_1},\ldots,X_{t_r}\}.
	\end{align*}
	Taking expectations on both sides of the above identity and using the fact that $X_1,\ldots,X_m$ are i.i.d. completes the proof.
\end{proof}

\section{General formula for the bias
	%Unbiasedness of extended $m$th Gini index estimator
}\label{general_bias}

In this section, we establish in Corollary~\ref{def-bias} a general formula for the bias of the linear order-statistic inequality index estimator $\widehat{I}_{m}$, $m\geqslant 2$, defined by:
\begin{align}\label{estimator}
	\widehat{I}_{m}
	=
		\dfrac{\displaystyle
			\binom{n}{m}^{-1} \,
			\sum_{\substack{{\bf i}=
					(i_1,\ldots,i_m)\in\mathbb{N}^m: \\[0,1cm]
					1\leqslant i_1<\cdots< i_m\leqslant n}} \
			\sum_{k=1}^{m}
			a_k
			X_{k:{\bf i}}
		}{\displaystyle m\overline{X}} \, 
		\mathds{1}_{\{\sum_{i=1}^{n}X_i>0\}}, 
\end{align}
where $\overline{X}=(1/n)\sum_{i=1}^{n}X_i$ is the sample mean,  $X_{1:\mathbf{i}}=\min\{X_{i_1},\ldots,X_{i_m}\}\leqslant\cdots\leqslant X_{m:\mathbf{i}}=\max\{X_{i_1},\ldots,X_{i_m}\}$ are the order statistics of the i.i.d.\ sample $X_{i_1},\ldots,X_{i_m}$, and $\widehat{I}_m$ is based on $n\geqslant m$ i.i.d. observations.

{
\begin{remark}\label{rem-L-stat}
The estimator $\widehat{I}_m$ defined in \eqref{estimator} admits an equivalent representation as a ratio of an L-statistic to the sample mean. Specifically, the numerator of $\widehat{I}_m$ satisfies
\begin{align}\label{L-stat-rep}
	\binom{n}{m}^{-1}
	\!\!\!\!\!
	\sum_{\substack{{\bf i}\in\mathbb{N}^m:\\ 1\leqslant i_1<\cdots<i_m\leqslant n}}
	\!\!\!\!\!
	\frac{1}{m}\sum_{k=1}^{m} a_k\, X_{k:{\bf i}}
	\;=\;
	\frac{1}{n}\sum_{r=1}^{n} X_{r:n}
	\sum_{s=1}^{\min(m,r)} a_s\,
	\frac{\dbinom{r-1}{s-1}\dbinom{n-r}{m-s}}{\dbinom{n-1}{m-1}},
\end{align}
where $X_{1:n}\leqslant\cdots\leqslant X_{n:n}$ denote the order statistics of the full sample of size $n$.
Therefore, $\widehat{I}_m$ can be expressed as
\begin{align}\label{Ihat-L-exact}
	\widehat{I}_m
	=
	\frac{1}{n\,\overline{X}}
	\sum_{r=1}^{n} X_{r:n}
	\sum_{s=1}^{\min(m,r)} a_s\,
	\frac{\dbinom{r-1}{s-1}\dbinom{n-r}{m-s}}{\dbinom{n-1}{m-1}},
\end{align}
which can be computed exactly in $O(n\log n + nm)$ time, without any subsampling or Monte Carlo approximation.
\end{remark}
}

\begin{theorem}\label{main-theorem}
	Let $X_1,\ldots,X_m$ be i.i.d. copies of a non-negative, non-degenerate random variable $X$ with finite, strictly positive mean $\mu$.
	For any $m\geqslant 2$, the following holds:
	\begin{align*}	
		\mathbb{E}\left[\widehat{I}_m\right]
		=
		{1\over m}
		\sum_{k=1}^{m}
		a_k
		\sum_{r=k}^{m}
		\binom{m}{r}
		(-1)^{r-k}
		\binom{r-1}{k-1}
		\mathbb{E}\left[
		\dfrac{X_{r:r}
		}{\overline{X}}
		\right].
	\end{align*}
	%
	%		
	%	\begin{align*}
		%		\mathbb{E}[	\widehat{IG}_m(j,k,p)]
		%		&=
		%		{1\over m}
		%		\sum_{r=j}^{m}
		%		\binom{m}{r}
		%		\Bigg[
		%		(-1)^{r-k}
		%		\binom{r-1}{k-1}
		%		\mathds{1}_{\{r\geqslant k\}}
		%		-
		%		(-1)^{r-j}
		%		\binom{r-1}{j-1}
		%		\Bigg]
		%		\\[0,2cm]
		%		&\times
		%		n
		%		\int_0^\infty
		%		\int_0^\infty 
		%		\left\{
		%		\mathscr{L}_{X^p}^{r}(z)
		%		-
		%		\mathbb{E}^r\left[
		%		\mathds{1}_{\{X\leqslant t^{1/p}\}}
		%		\exp\left(-X^p z\right)
		%		\right]
		%		\right\}
		%		{\rm d}t  \,
		%		\mathscr{L}_{X^p}^{n-r}(z)
		%		{\rm d}z,
		%	\end{align*}
	%	\begin{align*}
		%		IG_m(j,k,p)
		%		=
		%		{1\over m} 
		%			\displaystyle
		%			\sum_{r=j}^{m}
		%			\binom{m}{r}
		%			\left[
		%			(-1)^{r-k}\binom{r-1}{k-1}\mathds{1}_{\{r\geqslant k\}}
		%			-
		%			(-1)^{r-j}\binom{r-1}{j-1}
		%			\right]
		%		\dfrac{\displaystyle
			%			\int_0^\infty
			%			\left[1-F(t^{1/p})^r\right] \mathrm dt
			%		}
		%		{\displaystyle
			%			\int_0^\infty 
			%			\left[1-F(t^{1/p})\right]
			%			{\rm d}t
			%		}.
		%	\end{align*}
	%	where $\mathscr{L}_F(z)=\int_0^\infty \exp(-zx){\rm d}F(x)$ is the Laplace transform associated with distribution $F$, under the assumption that the relevant expectations and improper integrals converge.
\end{theorem}
\begin{proof}
	%[Proof of Theorem \ref{main-theorem}]
	%	{\color{black} Applying the classical result
		%		$\int_{0}^\infty \exp(-w z){\rm d}z={1/ w}$, $w>0$,
		%		and setting $w=\sum_{i=1}^{n}X_i$, we derive}
	%	\begin{multline}\label{exp-1}
		%		\mathbb{E}\left[
		%		\dfrac{\displaystyle
			%			\sum_{\substack{{\bf i}=
					%					(i_1,\ldots,i_m)\in\mathbb{N}^m: \\[0,1cm]
					%					1\leqslant i_1<\cdots< i_m\leqslant n}}
			%			\left(
			%			X_{k:{\bf i}}
			%			-
			%			X_{j:{\bf i}}
			%			\right)
			%		}{\displaystyle \sum_{i=1}^{n}X_i}
		%		\right]
		%		%		=
		%		%		\sum_{\substack{{\bf i}=
				%				%				(i_1,\ldots,i_m)\in\mathbb{N}^m: \\[0,1cm]
				%				%				1\leqslant i_1<\cdots< i_m\leqslant n}}
		%		%		\mathbb{E}\left[X_{k:{\bf i}}\int_0^\infty\exp\left\{-\left(\sum_{i=1}^{n}X_i\right)z\right\}{\rm d}z\right]
		%		%		\\[0,2cm]
		%		%		-
		%		%		\sum_{\substack{{\bf i}=
				%				%				(i_1,\ldots,i_m)\in\mathbb{N}^m: \\[0,1cm]
				%				%				1\leqslant i_1<\cdots< i_m\leqslant n}}
		%		%		\mathbb{E}\left[X_{j:{\bf i}}\int_0^\infty\exp\left\{-\left(\sum_{i=1}^{n}X_i\right)z\right\}{\rm d}z\right]
		%		%		\\[0,2cm]
		%		=
		%		\sum_{\substack{{\bf i}=
				%				(i_1,\ldots,i_m)\in\mathbb{N}^m: \\[0,1cm]
				%				1\leqslant i_1<\cdots< i_m\leqslant n}} \
		%		\int_0^\infty
		%		\mathbb{E}\left[(X_{k:{\bf i}}-X_{j:{\bf i}})
		%		\exp\left\{-\left(\sum_{i=1}^{n}X_i\right)z\right\}
		%		\right]
		%		{\rm d}z,
		%	\end{multline}
	%	where Tonelli's Theorem justifies the interchange of integrals.
	{\color{black} Making use of the identities provided in} \cite[][Proposition 10]{Salama-Koch2020}:
	\begin{align*}
		X_{k:{\bf i}}
		=
		\sum_{r=k}^{m}
		(-1)^{r-k}
		\binom{r-1}{k-1}
		\sum_{\substack{{\bf t}=(t_1,\ldots,t_r)\in\mathbb{N}^r: \\[0,1cm]
				1\leqslant t_1<\cdots<t_r\leqslant m}}
		X_{r:{\bf t}},
	\end{align*}
	we can write
	\begin{align}\label{exp-2}
		\mathbb{E}\left[
		\dfrac{\displaystyle
			\sum_{\substack{{\bf i}=
					(i_1,\ldots,i_m)\in\mathbb{N}^m: \\[0,1cm]
					1\leqslant i_1<\cdots< i_m\leqslant n}} \
			\sum_{k=1}^{m}
			a_k
			X_{k:{\bf i}}
		}{\overline{X}}
		\right]
		=
		\sum_{\substack{{\bf i}=
				(i_1,\ldots,i_m)\in\mathbb{N}^m: \\[0,1cm]
				1\leqslant i_1<\cdots< i_m\leqslant n}}
		\sum_{k=1}^{m}
		a_k
		\sum_{r=k}^{m}
		(-1)^{r-k}
		\binom{r-1}{k-1}
		\sum_{\substack{{\bf t}=(t_1,\ldots,t_r)\in\mathbb{N}^r: \\[0,1cm]
				1\leqslant t_1<\cdots<t_r\leqslant m}} 
		\mathbb{E}\left[
		\dfrac{X_{m:{\bf t}}}{\overline{X}}
		\right].
	\end{align}
	Using the fact that $X_1,\ldots,X_m$ are i.i.d., the identity in \eqref{exp-2} takes the simplified form:
	%, we have
	%	\begin{align*}
		%		X_{m:{\bf t}}
		%		&\exp\left\{-\left(\sum_{i=1}^{r}X_i\right)z\right\}
		%		=
		%		\max\{X_{t_1},\ldots,X_{t_r}\}
		%		\exp\left\{-\left(\sum_{i=1}^{r}X_i\right)z\right\}
		%		\\[0,2cm]
		%		&
		%		\stackrel{d}{=}
		%		\max\{X_{1},\ldots,X_{r}\}
		%		\exp\left\{-\left(\sum_{i=1}^{r}X_i\right)z\right\}
		%		=
		%		X_{r:r}
		%		\exp\left\{-\left(\sum_{i=1}^{r}X_i\right)z\right\}
		%	\end{align*}
	%	and
	%	\begin{align*}
		%		X_{m:{\bf u}}
		%		&\exp\left\{-\left(\sum_{i=1}^{s}X_i\right)z\right\}
		%		=
		%		\max\{X_{u_1},\ldots,X_{u_s}\}
		%		\exp\left\{-\left(\sum_{i=1}^{s}X_i\right)z\right\}
		%		\\[0,2cm]
		%		&
		%		\stackrel{d}{=}
		%		\max\{X_{1},\ldots,X_{s}\}
		%		\exp\left\{-\left(\sum_{i=1}^{s}X_i\right)z\right\}
		%		=
		%		X_{s:s}
		%		\exp\left\{-\left(\sum_{i=1}^{s}X_i\right)z\right\},
		%	\end{align*}
	%	where $\stackrel{d}{=}$ denotes equality in distribution of random variables. Applying the aforementioned identities to \eqref{exp-2} and leveraging independence, the expression simplifies to
	%	{\small
		\begin{align*}
			%		&=	
			%		\binom{n}{m}
			%		\!\!
			%		\sum_{r=k}^{m}
			%		(-1)^{r-k}
			%		\binom{r-1}{k-1} 		\!\!
			%		\binom{m}{r} 		\!\!
			%		\int_0^\infty 		\!\!
			%		\mathbb{E}\left[ 		\!
			%		X_{r:r} 		\!
			%		\exp\left\{		\!-		\!\left(\sum_{i=1}^{r}X_i\right)		\!z		\!\right\} 		\!
			%		\mathbb{E}\left[
			%		\exp\left\{-		\!\left(\sum_{i=r+1}^{n}X_i\right)		\!z		\!\right\}
			%				\!
			%		\right]
			%				\!
			%		\right] 				\!\!
			%		{\rm d}z
			%		\\[0,2cm]
			%		&-
			%		\binom{n}{m}
			%				\!\!
			%		\sum_{s=j}^{m}
			%		(-1)^{s-j}
			%		\binom{s-1}{j-1} 		\!\!
			%		\binom{m}{s} 		\!\!
			%		\int_0^\infty 		\!\!
			%		\mathbb{E}\left[ 		\!
			%		X_{s:s} 		\!
			%		\exp\left\{-\left(\sum_{i=1}^{s}X_i\right)		\!z		\!\right\}
			%				\!
			%		\mathbb{E}
			%		\left[
			%		\exp\left\{		\!-		\!\left(\sum_{i=s+1}^{n}X_i\right)		\!z		\!\right\} 		\!
			%		\right] 		\!
			%		\right] 		\!\!
			%		{\rm d}z
			%		\\[0,2cm]
			%		&=
			\mathbb{E}\left[
			\dfrac{\displaystyle
				\sum_{\substack{{\bf i}=
						(i_1,\ldots,i_m)\in\mathbb{N}^m: \\[0,1cm]
						1\leqslant i_1<\cdots< i_m\leqslant n}}
				\sum_{k=1}^{m}
				a_k
				X_{k:{\bf i}}
			}{\overline{X}}
			\right]
			=
			\binom{n}{m}
			\sum_{k=1}^{m}
			a_k
			\sum_{r=k}^{m}
			\binom{m}{r}
			(-1)^{r-k}
			\binom{r-1}{k-1}
			\mathbb{E}\left[
			\dfrac{X_{r:r}
			}{ \overline{X}}
			\right].
		\end{align*}
		%}
	Finally, the result follows immediately from definition \eqref{estimator} of $\widehat{I}_m$ and the preceding identity.
\end{proof}

By combining Theorem~\ref{main-theorem} and Proposition~\ref{ext-gini-index-0}, we obtain:
\begin{corollary}\label{def-bias}
	For any $m\geqslant 2$, the bias of $\widehat{I}_m$ relative to $I_m$, denoted by ${\rm Bias}(\widehat{I}_m,I_m)$, can be written as
	\begin{align*}
		{\rm Bias}(\widehat{I}_m,I_m)
		=
		{1\over m} 
		\sum_{k=1}^{m}
		a_k
		\sum_{r=k}^{m}
		\binom{m}{r} 
		(-1)^{r-k}
		\binom{r-1}{k-1}
		\Delta_{n,r},
	\end{align*}
	where 
	\begin{align}\label{def-delta}
		\Delta_{n,r}
		\equiv
		\mathbb{E}
		\left[
		\dfrac{X_{r:r}
		}{
			\overline{X}
		}
		\right]
		-
		\dfrac{ 
			\mathbb{E}
			\left[
			X_{r:r}
			\right] 
		}
		{
			\mu
		}.
	\end{align}
\end{corollary}

Under suitable integrability conditions, the next result and Corollary \ref{def-bias} yield (not necessarily closed-form) expressions for the bias over general non-negative-support populations. 
%We strongly believe that, within the class of absolutely continuous distributions supported on $[0,\infty)$ and having finite and positive mean, the gamma distribution is the only one for which the bias vanishes.
\begin{proposition}\label{prop-main}
	For each $n\geqslant r$ the factor $\Delta_{n,r}$ in \eqref{def-delta} can be characterized as
	\begin{align*}
		\Delta_{n,r}
		=
		\int_0^\infty
		\left\{
		\int_0^\infty 
		\left[
		1
		-
		F_z(t)
		\right]
		n
		\mathscr{L}^{n}(z)
		{\rm d}z  \,
		-
		\dfrac{ 
			\left[1-F^r(t)\right]
		}{\mu} 
		\right\}
		{\rm d}t,
	\end{align*}  
	provided that the Laplace transform $\mathscr{L}(z)=\mathbb{E}[\exp(-zX)]$ of $X$ and the associated improper integrals exist. In the above, $F$ is the CDF of $X$ and $F_z$ is the CDF defined as
	\begin{align*}
		F_z(t)
		\equiv
		\dfrac{
			\mathbb{E}^r\left[
			\mathds{1}_{\{X\leqslant t\}}
			\exp\left(-X z\right)
			\right]
		}{\mathscr{L}^{r}(z)}, \quad 
		t>0, \, z>0.
	\end{align*}
\end{proposition}
\begin{proof}
	The result is an immediate consequence of the identities
	$
	\int_{0}^{\infty} \exp{(-wz)}\,\mathrm{d}z={1}/{w},
	\ w>0,
	$
	applied with $w=\sum_{i=1}^n X_i$, and
	$
	X_{r:r}=\int_{0}^{\infty}\mathds{1}_{\{X_{r:r}>t\}}\,\mathrm{d}t,
	$
	combined with Tonelli's theorem and the i.i.d. assumption on $X_1,\ldots,X_n$.
\end{proof}

\section{Asymptotic unbiasedness}\label{Asymptotic unbiasedness}

In this section, we show that, {under moment and lower-tail conditions}, the bias of the linear order-statistic inequality index estimator \eqref{estimator}, as characterized in Corollary~\ref{def-bias}, vanishes as the sample size $n$ increases.

%%{\color{red}The analytical mechanism behind this result is made transparent by the Laplace %%%%%representation of Proposition~\ref{prop-main}: as noted in Remark~\ref{rem:prop-main-roles}(ii), the %%%measure $n\mathscr{L}^n(z)\,\mathrm{d}z$ concentrates near the origin as $n\to\infty$, driving $%%%\Delta_{n,r} \to 0$. The proof below makes this precise via a uniform integrability argument.}

\begin{proposition}\label{prop-main-1}
	{Fix \(p>1\). Assume that there exists \(\varepsilon>0\) such that
		$
		\mathbb{E}[X^{p(1+\varepsilon)}]<\infty,
		$
		and that there exist constants \(C>0\), \(\gamma>0\), \(t_0>0\) such that
		$
		\sup_{n\geqslant 1}\mathbb{P}(\overline{X} \leqslant t)
		\leqslant C t^{\gamma}, \ t\in(0,t_0],
		$
		with
		$
		\gamma > {p(1+\varepsilon)}/{\varepsilon}.
		$}
	Then, for any \(m\geqslant 2\), we have
$$\lim_{n\to\infty}{\rm Bias}(\widehat{I}_m,I_m)=0.$$
\end{proposition}
\begin{proof}
	Using the weak law of large numbers,
\begin{align}\label{as-convergence}
	\overline{X}\stackrel{\text{P}}{\longrightarrow} \mu,
	\quad \text{as} \ n\to\infty,
\end{align}
where $``\stackrel{\text{P}}{\longrightarrow}"$ denotes convergence in probability.

{
We show that \(\{{X_{r:r}}/{\overline{X}}\}_{n\geqslant 1}\) is uniformly integrable. 
Indeed, since \(X_{r:r}^p \leqslant \sum_{i=1}^r X_i^p\), $p>1$, and $X_1,X_2,\ldots$ are i.i.d., we have
\begin{align}\label{ineq-id-0}
	\mathbb{E}\!\left[\left(\frac{X_{r:r}}{\overline{X}}\right)^p\right]
	\leqslant 
	r\,\mathbb{E}\!\left[\left(\frac{X_1}{\overline{X}}\right)^p\right].
\end{align}
By Hölder’s inequality, for any \(\varepsilon>0\),
\begin{equation}\label{ineq-id-1}
	\mathbb{E}\!\left[\left(\frac{X_1}{\overline{X}}\right)^p\right]
	\leqslant
	\left\{\mathbb{E}[X^{p(1+\varepsilon)}]\right\}^{1/(1+\varepsilon)}
	\left\{\mathbb{E}\!\left[\overline{X}^{\, -p(1+\varepsilon)/\varepsilon}\right]\right\}^{\varepsilon/(1+\varepsilon)}.
\end{equation}

Using the identity \(\mathbb{E}[Z^{-s}]=s\int_0^\infty t^{-s-1}\mathbb{P}(Z\leqslant t)\,{\rm d}t\) (valid for \(Z>0\) a.s. and $s>0$), together with the bound
\(\mathbb{P}(\overline{X}\leqslant t)\le C t^\gamma\) for \(t\in(0,t_0]\), we obtain
\[
\mathbb{E}\!\left[\overline{X}^{\, -p(1+\varepsilon)/\varepsilon}\right]
\leqslant
C\gamma\int_0^{t_0} t^{\gamma - p(1+\varepsilon)/\varepsilon -1}\,{\rm d}t
+
\gamma
\int_{t_0}^{\infty} t^{-p(1+\varepsilon)/\varepsilon -1}\,{\rm d}t.
\]
Since \(\gamma > p(1+\varepsilon)/\varepsilon\), both integrals are finite, and therefore
\begin{equation}\label{ineq-id-2}
	\sup_{n\geqslant 1}\mathbb{E}\!\left[\overline{X}^{\, -p(1+\varepsilon)/\varepsilon}\right] < \infty.
\end{equation}

Combining \eqref{ineq-id-0}, \eqref{ineq-id-1} and \eqref{ineq-id-2} and using \(\mathbb{E}[X^{p(1+\varepsilon)}]<\infty\), we conclude that
\[
\sup_{n\geqslant 1}
\mathbb{E}\!\left[\left(\frac{X_{r:r}}{\overline{X}}\right)^p\right] < \infty.
\]
Since \(p>1\), it follows that \(\{{X_{r:r}}/{\overline{X}}\}_{n\geqslant 1}\) is uniformly integrable.
}

		Finally, combining \eqref{as-convergence}, the uniform integrability of \(\{{X_{r:r}/
		}{\overline{X}}\}_{n\geqslant 1}\), and standard convergence theorems (e.g., convergence in probability plus uniform integrability implies convergence in $L^1$), we obtain
	\begin{align*}
		\lim_{n\to\infty}
		\mathbb{E}
		\left[
		\dfrac{X_{r:r}
		}{\overline{X}}
		\right]
		=
		\dfrac{ 
			\mathbb{E}
			\left[
			X_{r:r}
			\right] 
		}
		{
			\mu
		}.
	\end{align*}
	Applying Corollary~\ref{def-bias} completes the proof.
	\end{proof}

{
\begin{remark}
	Let \(X \sim \mathrm{Gamma}(\alpha,\lambda)\) with \(\alpha>0\) and \(\lambda>0\).
	Then the moment condition \(\mathbb{E}[X^{p(1+\varepsilon)}]<\infty\) holds for any \(p>1\) and \(\varepsilon>0\), since the gamma distribution has finite moments of all positive orders.
	Moreover, since \(\overline{X} \sim \mathrm{Gamma}(n\alpha,n\lambda)\), its distribution near zero satisfies
	$
	\mathbb{P}(\overline{X} \leqslant t) \leqslant C t^{n\alpha}, \ t\in(0,t_0],
	$
	for some constants \(C=\sup_{n\geqslant 1}(\lambda n)^{n\alpha}/\Gamma(n\alpha+1)>0\) and \(0<t_0\leqslant 1\). In particular,
	$
	\sup_{n\geqslant1}\mathbb{P}(\overline{X} \leqslant t) \leqslant C t^{\alpha}, \ t\in(0,t_0]$,
	so that the lower-tail condition is satisfied with exponent \(\gamma=\alpha\).
	Therefore, the assumptions of the Proposition \ref{prop-main-1} hold whenever
	$
	\alpha > {p(1+\varepsilon)}/{\varepsilon}.
	$
%	
%	We note that this condition is sufficient but not optimal; sharper arguments allow one to replace it by the weaker requirement \(\alpha > p\).
\end{remark}

}

\section{Unbiasedness in gamma populations}\label{Unbiasedness_gamma}

{
In this section, we establish one of the main results of the paper, namely the unbiasedness of the estimator \(\widehat{I}_m\) for gamma populations. The proof relies on three key ingredients: the independence between the vector of normalized independent gamma variables and their total sum, the homogeneity of the maximum function, and Corollary~\ref{def-bias}.

\begin{theorem}
Let $X_1,\ldots,X_n$ be a random sample from $X\sim\mathrm{Gamma}(\alpha,\lambda)$. Then $\widehat{I}_m$ is an unbiased estimator of $I_m$, that is,
\begin{equation}\label{bias=0}
	\mathrm{Bias}(\widehat{I}_m,\,I_m)=0,
	\quad m\geqslant 2.
\end{equation}
\end{theorem}
}
\begin{proof}
It is well known \cite[see][Theorem~2.1]{Mosimann1962} that
\[
\left(
\frac{X_1}{\sum_{i=1}^n X_i},
\ldots,
\frac{X_n}{\sum_{i=1}^n X_i}
\right)
\sim \mathrm{Dirichlet}(\alpha,\ldots,\alpha),
\]
and that this random vector is independent of $\sum_{i=1}^n X_i$.

For $1\leqslant r\leqslant n$, define
\[
g(y_1,\ldots,y_n)=\max\{y_1,\ldots,y_r\}.
\]
Using the homogeneity of $g$ and the above independence, we obtain
\begin{align*}
	\mathbb{E}[X_{r:r}]
	=
	\mathbb{E}\!\left[g(X_1,\ldots,X_n)\right] 
	&=
	\mathbb{E}\!\left[
	g\!\left(
	\frac{X_1}{\sum_{i=1}^n X_i},
	\ldots,
	\frac{X_n}{\sum_{i=1}^n X_i}
	\right)
	\sum_{i=1}^n X_i
	\right] \\[0,2cm]
	&=
	\mathbb{E}\!\left[\sum_{i=1}^n X_i\right]
	\mathbb{E}\!\left[
	g\!\left(
	\frac{X_1}{\sum_{i=1}^n X_i},
	\ldots,
	\frac{X_n}{\sum_{i=1}^n X_i}
	\right)
	\right] \\[0,2cm]
	&=
	n\mu\,
	\mathbb{E}\!\left[
	g\!\left(
	\frac{X_1}{\sum_{i=1}^n X_i},
	\ldots,
	\frac{X_n}{\sum_{i=1}^n X_i}
	\right)
	\right]
	=
	\mu\,
	\mathbb{E}\!\left[
	\frac{X_{r:r}}{\overline{X}}
	\right].
\end{align*}

Therefore, the quantity $\Delta_{n,r}$ defined in \eqref{def-delta} satisfies
\[
\Delta_{n,r}=0,
\quad n\geqslant r.
\]
An application of Corollary~\ref{def-bias} then yields
\eqref{bias=0}.
\end{proof}

{
\begin{remark}
The gamma distribution plays a fundamental role in the unbiasedness result established in this paper. The proof relies on the independence between the vector of normalized independent gamma variables and their total sum, together with the homogeneity of the maximum function, which yields an exact cancellation of the bias. To the best of our knowledge, this independence structure is specific to the gamma family and does not hold for general distributions. Consequently, the proof presented here of the exact unbiasedness of \(\widehat{I}_m\) is specific to gamma populations.

Although related ideas may be considered for other distributional families, such as certain generalized gamma or Dirichlet-type models arising from gamma-based constructions, these typically do not preserve the independence between ratios and the total sum, and therefore do not directly extend the present result.
\end{remark}
}

%{\color{black}
	%We now proceed by applying Theorem \ref{main-theorem} to derive an explicit expression for the expectation of the estimator $\widehat{IG}_m(j,k)$ in gamma populations.}
%
%\begin{corollary}\label{corollary-main}
%	Let $X_1, X_2,\ldots, X_m$ be independent copies of $X\sim\text{Gamma}(\alpha,\lambda)$. For any $1\leqslant {\color{black}j < k}\leqslant m$, we have:
%	\begin{align*}
	%		\mathbb{E}[\widehat{IG}_m(j,k)]
	%		&=
	%		{1\over \alpha m}
	%		%
	%		\sum_{r=k}^{m}
	%		(-1)^{r-k}
	%		\binom{r-1}{k-1}
	%		\binom{m}{r}
	%		\int_0^\infty 
	%		\left\{
	%		1
	%		-
	%		{\gamma^r(\alpha,v)\over\Gamma^r(\alpha)}
	%		\right\}
	%		{\rm d}v  \,
	%		\nonumber
	%		\\[0,2cm]
	%		&-
	%		{1\over \alpha m}
	%		%
	%		\sum_{s=j}^{m}
	%		(-1)^{s-j}
	%		\binom{s-1}{j-1}
	%		\binom{m}{s}
	%		\int_0^\infty 
	%		\left\{
	%		1
	%		-
	%		{\gamma^s(\alpha,v)\over\Gamma^s(\alpha)}
	%		\right\}
	%		{\rm d}v
	%		=
	%		{IG}_m(j,k),
	%	\end{align*}
%	where ${IG}_m(j,k)$ is the {\color{black}extended $m$th Gini index} given in Proposition \ref{ext-gini-index}. Thus, the estimator $	\widehat{IG}_m(j,k)$ is unbiased
%	for gamma populations.	
%\end{corollary}

\begin{remark}
	Observe that \eqref{bias=0} extends the earlier results of \cite{Deltas2003}, \cite{Baydil2025}, \cite{Vila2025}, and \cite{Vila2026}. Furthermore, since the index $I_m$
	yields particular cases of the $S$-Gini index (see Remark \ref{rem-1}), Equation \eqref{bias=0} implies that, for these cases, the $S$-Gini estimator is unbiased for gamma populations, a result that, to the best of our knowledge, has not been previously reported in the literature.
\end{remark}

%\begin{remark}
%	Due to the scale invariance of the extended Gini index estimator $\widehat{IG}_m(j,k)$, its expectation $\mathbb{E}[\widehat{IG}_m(j,k)]$ is independent of the rate parameter $\lambda$, as asserted in  Corollary  \ref{corollary-main}.
%\end{remark}

%Setting $j=1$ and $k=m$ into Corollary \ref{corollary-main}, from Newton's binomial Theorem it follows that
%\begin{proposition}\label{prop-main}
%	Let $X_1, X_2,\ldots, X_m$ be independent copies of $X\sim\text{Gamma}(\alpha,\lambda)$. We have:
%	\begin{align*}
	%		\mathbb{E}[\widehat{IG}_m]
	%		\equiv
	%		\mathbb{E}[\widehat{IG}_m(1,m)]
	%		&=
	%		{1\over \alpha m}
	%		\left[
	%		\int_0^\infty 
	%		\left\{
	%		1
	%		-
	%		{\gamma^m(\alpha,v)\over\Gamma^m(\alpha)}
	%		\right\}
	%		{\rm d}v  \,
	%		\nonumber
	%		-
	%		%
	%		\int_0^\infty 
	%		\left\{
	%		1
	%		-
	%		{\gamma(\alpha,v)\over\Gamma(\alpha)}
	%		\right\}^m
	%		{\rm d}v
	%		\right]
	%		\\[0,2cm]
	%		&=
	%		{IG}_m(1,m)
	%		\equiv 
	%		IG_m,
	%	\end{align*}
%	where $IG_m$ is the $m$th Gini index given in Remark \ref{rem-i} and $\widehat{IG}_m$ is its  corresponding estimator (initially introduced by \cite{Vila2025}) given in  Remark \ref{rem-gini-index}.
%\end{proposition}
%
%
%\begin{remark}
%	Note that the result of Proposition \ref{prop-main} was previously established in \cite{Vila2025}.
%\end{remark}

\section{Monte Carlo simulation study}\label{simulation study}

{
This section presents Monte Carlo simulation results assessing the finite-sample performance of the linear order-statistic inequality index estimator $\widehat I_m$ defined in \eqref{estimator}. The simulations are designed to complement the theoretical results of Sections \ref{general_bias}--\ref{Unbiasedness_gamma} by illustrating the bias and mean squared error behavior of the estimator under gamma and non-gamma populations. To place the proposed nonparametric estimator in context, we also include, for each distributional scenario, a parametric plug-in competitor $\widetilde{I}_5^{\mathrm{MLE}}$ obtained by fitting the correct distributional family to each sample by maximum likelihood and then computing $I_5$ via quadrature from the fitted distribution.

{Throughout this section, we fix $m=5$ and set the weight vector $(a_1,\ldots,a_5)=(-1,0,0,0,1)$, which corresponds to the 5th Gini index $IG_5$ defined in Table~\ref{Table-1}. This choice yields the corresponding population index
\[
I_5
=
\frac{\mathbb{E}[X_{5:5}]-\mathbb{E}[X_{1:5}]}{5\mu},
\]
where $\mu = \mathbb{E}[X]$ and $X_{1:5}\leqslant \cdots \leqslant X_{5:5}$ denote the order statistics of an i.i.d.\ sample of size $5$ from a non-negative distribution. For each data-generating process, the true population value $I_5^{\mathrm{true}}$ is computed via numerical integration using the integral representation (see the Integral representation property in Section~\ref{Basic properties}):
\begin{align}\label{eq:Im-quadrature}
I_5^{\mathrm{true}}
=
\frac{1}{\mu}\int_0^1 w_5(u)\,Q_X(u)\,\mathrm{d}u,
\quad
w_5(u) = u^4-(1-u)^4,
\end{align}
where $Q_X$ denotes the quantile function of $X$. This integral is evaluated by Gauss--Legendre quadrature (using \texttt{integrate} in \textsf{R}), which yields numerically stable benchmark values without Monte Carlo error.}

The simulations are carried out for sample sizes $n\in\{10,20,30,50\}$. For each distribution and each $n$, the expectations of both estimators are approximated using $R_{\mathrm{MC}}=2{,}000$ Monte Carlo replications. {Within each replication, $\widehat{I}_5$ is computed exactly using the L-statistic representation \eqref{Ihat-L-exact}, which requires only sorting the sample and evaluating the combinatorial weights $\binom{r-1}{s-1}\binom{n-r}{m-s}/\binom{n-1}{m-1}$ for $s\in\{1,5\}$.} The MLE plug-in $\widetilde{I}_5^{\mathrm{MLE}}$ is obtained within each replication by fitting the distribution parameters by maximum likelihood and computing $I_5$ via quadrature from the fitted quantile function. Specifically, for gamma populations the shape and rate parameters are estimated using the standard MLE; for lognormal populations the MLE is closed-form ($\hat{\mu}_{\log}=n^{-1}\sum\log x_i$, $\hat{\sigma}_{\log}=\sqrt{n^{-1}\sum(\log x_i-\hat{\mu}_{\log})^2}$); for Weibull populations the shape and scale are estimated by maximum likelihood; and for Lomax populations with scale fixed at unity the shape MLE is $\hat{\alpha}=n\big/\sum_{i=1}^{n}\log(1+x_i)$, which has a closed form. {The true benchmark $I_5^{\mathrm{true}}$ is obtained by quadrature as described in~\eqref{eq:Im-quadrature}.}

For each $(\text{distribution},n)$ pair, we report for both $\widehat I_5$ and $\widetilde{I}_5^{\mathrm{MLE}}$:
\begin{itemize}
	\item the bias $\mathbb{E}[\widehat I_5]-I_5^{\mathrm{true}}$ (resp.\ $\mathbb{E}[\widetilde{I}_5^{\mathrm{MLE}}]-I_5^{\mathrm{true}}$),
	\item the root mean squared error
	\[
	\mathrm{RMSE}(\widehat I_5)
	=
	\left\{
	\mathbb{E}\bigl[(\widehat I_5-I_5^{\mathrm{true}})^2\bigr]
	\right\}^{1/2},
	\quad
	\mathrm{RMSE}(\widetilde{I}_5^{\mathrm{MLE}})
	=
	\left\{
	\mathbb{E}\bigl[(\widetilde{I}_5^{\mathrm{MLE}}-I_5^{\mathrm{true}})^2\bigr]
	\right\}^{1/2}.
	\]
\end{itemize}

The following positive-support distributions are considered:
\begin{enumerate}
	\item {Gamma$(\alpha,1)$}: {four shape parameters are examined, $\alpha\in\{0.5,\,2,\,5,\,10\}$,} with unit rate. In all cases, the mean is $\mu=\alpha$. {This range covers highly skewed ($\alpha=0.5$), moderately skewed ($\alpha=2$), mildly skewed ($\alpha=5$), and near-symmetric ($\alpha=10$) shapes, in order to examine the unbiasedness result across varying degrees of skewness within the gamma family.}
	\item {Lognormal$(0,\sigma)$}: {with scale parameters $\sigma\in\{0.5,\,1\}$, yielding $\mu=\exp(\sigma^2/2)$. The case $\sigma=0.5$ represents moderate skewness, whereas $\sigma=1$ corresponds to heavier tails.}
	\item {Weibull$(k,1)$}: {with shape $k\in\{0.8,\,1.6\}$ and scale $\lambda=1$,} for which $\mu=\Gamma(1+1/k)$. {The shape $k=0.8$ gives a heavy-tailed Weibull (decreasing hazard), while $k=1.6$ gives a light-tailed case (increasing hazard).}
	\item {Lomax$(\alpha,1)$}: {with shape $\alpha\in\{3,\,5\}$} and scale $\lambda=1$, implying $\mu=1/(\alpha-1)$. {Larger $\alpha$ corresponds to lighter tails.}
\end{enumerate}

These distributions were selected to represent {a wide range of tail behaviors and skewness levels:} light-tailed (gamma), moderately skewed (Weibull), and heavy-tailed (lognormal and Lomax) cases. {In particular, the multiple gamma scenarios allow direct verification that unbiasedness holds for all shape parameter values, while the non-gamma scenarios illustrate how tail heaviness amplifies finite-sample bias.}

Table~\ref{tab:mc_results} reports the Monte Carlo results. {The true population values $I_5^{\mathrm{true}}$ reported in the table are obtained by numerical quadrature via~\eqref{eq:Im-quadrature}.} The simulation results provide numerical confirmation of the theoretical findings established in Sections~\ref{Asymptotic unbiasedness} and~\ref{Unbiasedness_gamma}, and reveal important differences between the proposed U-statistic estimator and the MLE plug-in. For {all four gamma scenarios}, the empirical bias of $\widehat I_{\color{red}5}$ fluctuates tightly around zero and may change sign across $n$, which is expected due to the variance of the Monte Carlo bias estimates. {This behavior is consistent across $\alpha\in\{0.5,2,5,10\}$, suggesting that exact unbiasedness holds regardless of the gamma shape parameter.} In sharp contrast, the MLE plug-in $\widetilde{I}_5^{\mathrm{MLE}}$ exhibits systematic negative bias under all four gamma scenarios, ranging from approximately $-0.027$ at $n=10$ to $-0.003$ at $n=50$. In terms of RMSE, the two estimators are broadly comparable under gamma populations, with the MLE plug-in achieving slightly lower values due to the gain in efficiency from exploiting the parametric form.

From Table~\ref{tab:mc_results}, we observe that the U-statistic estimator exhibits increased finite-sample bias under non-gamma distributions. For the Lognormal$(0,1)$ and Lomax$(3,1)$ cases, the bias is negative and substantial for small $n$. Although the magnitude of the bias decreases as $n$ increases, it remains clearly distinguishable from zero even at $n=50$, particularly for the heavy-tailed Lomax distribution. {Comparing Lognormal$(0,0.5)$ with Lognormal$(0,1)$ illustrates that larger $\sigma$ amplifies the bias, as expected from the heavier right tail. Similarly, the Weibull$(0.8,1)$ case exhibits more pronounced bias than the Weibull$(1.6,1)$ case due to its heavier tail.} The Weibull$(1.6,1)$ distribution presents an intermediate behavior, with relatively small bias that fluctuates in sign across sample sizes.

Comparing the two estimators under non-gamma scenarios, the MLE plug-in $\widetilde{I}_5^{\mathrm{MLE}}$ has somewhat larger negative bias than $\widehat{I}_5$ at small $n$ for lognormal and Weibull populations, but achieves substantially lower RMSE by leveraging the correct parametric form. The improvement in RMSE is most dramatic for Lomax populations: at $n=10$, RMSE drops from $0.1224$ to $0.0445$ for Lomax$(3,1)$ and from $0.1036$ to $0.0209$ for Lomax$(5,1)$, reductions of approximately $64\%$ and $80\%$ respectively. This reflects the efficiency of the closed-form Lomax MLE $\hat{\alpha}=n/\sum\log(1+x_i)$. For lognormal and Weibull populations the RMSE gains are more modest (roughly $10$--$25\%$ at $n=10$) and diminish as $n$ grows. These results confirm that, while the MLE plug-in is more efficient when the parametric family is correctly specified, the proposed U-statistic estimator is robust to distributional misspecification, distribution-free, and exactly unbiased under the gamma family, properties the parametric plug-in does not share.

{
\medskip\noindent\textit{Validation of the bias decomposition and connection to Proposition~\ref{prop-main}.}
The Monte Carlo biases reported in Table~\ref{tab:mc_results} can be decomposed through Corollary~\ref{def-bias} into contributions from the individual factors $\Delta_{n,r}$. For the weight vector $(-1,0,0,0,1)$, the bias coefficients are $c_1=-1$, $c_2=2$, $c_3=-2$, $c_4=1$, $c_5=0$, so only the first four factors contribute:
\[
\mathrm{Bias}(\widehat{I}_5,I_5)
=-\Delta_{n,1}+2\Delta_{n,2}-2\Delta_{n,3}+\Delta_{n,4}.
\]
We estimated each $\Delta_{n,r}$ ($r=1,\ldots,4$) for the Lognormal$(0,1)$ and Lomax$(3,1)$ scenarios using an independent large auxiliary Monte Carlo simulation with $R=300{,}000$ replications. Table~\ref{tab:delta_validation} reports the estimated components and the resulting theoretical bias alongside the direct Monte Carlo bias from Table~\ref{tab:mc_results}. Several patterns are apparent. First, $\Delta_{n,1}$ is negligible across all cases (in absolute value below $0.003$), indicating that the minimum order statistic contributes almost no bias. Second, $\Delta_{n,2}$, $\Delta_{n,3}$, and $\Delta_{n,4}$ are all negative and increase in magnitude with $r$: upper order statistics are more sensitive to the random normalization by $\overline{X}$ than lower ones, and their contribution to the negative bias grows with rank. Third, the Lomax$(3,1)$ factors are substantially larger in magnitude than the Lognormal$(0,1)$ ones at the same $n$, consistent with the heavier upper tail of the Lomax distribution amplifying the normalization effect. Finally, the theoretical bias reconstructed from the decomposition closely tracks the direct Monte Carlo estimates throughout, confirming the accuracy of Corollary~\ref{def-bias}. Each factor $\Delta_{n,r}$ in Table~\ref{tab:delta_validation} is precisely the quantity that Proposition~\ref{prop-main} expresses as a double integral involving the Laplace transform $\mathscr{L}$ and the tilted CDF $F_z$. For any distribution with a numerically computable $\mathscr{L}$, the Proposition therefore provides a simulation-free route to these same values by two-dimensional quadrature, with no Monte Carlo noise. This dual computability, via direct Monte Carlo and via the Laplace integral, illustrates the practical role of Proposition~\ref{prop-main} as a tool for bias analysis under general non-negative populations.

%%The theoretical bias is $-\Delta_{n,1}+2\Delta_{n,2}-2\Delta_{n,3}+\Delta_{n,4}$; the MC bias is from %%Table~\ref{tab:mc_results} ($R_{\mathrm{MC}}=2{,}000$). Each $\Delta_{n,r}$ is also the value that %Proposition~\ref{prop-main} yields by numerical quadrature of the Laplace-transform integral.

%%The true value $I_5^{\mathrm{true}}$ is obtained by numerical quadrature via~\eqref{eq:Im-quadrature}. %%The MLE plug-in $\widetilde{I}_5^{\mathrm{MLE}}$ is computed by fitting the correct parametric family to %%5each sample by maximum likelihood and evaluating $I_5$ via quadrature from the fitted distribution.

\begin{table}[!ht]
	\centering
	\small
	\caption{Monte Carlo results for $\widehat I_5$ and $\widetilde{I}_5^{\mathrm{MLE}}$ ($m=5$, $(a_1,\ldots,a_5)=(-1,0,0,0,1)$): bias and RMSE based on $R_{\mathrm{MC}}=2{,}000$ replications.}
	\label{tab:mc_results}
	\begin{tabular}{lrr cc cc}
		\hline
		& & & \multicolumn{2}{c}{$\widehat{I}_5$ (U-statistic)} & \multicolumn{2}{c}{$\widetilde{I}_5^{\mathrm{MLE}}$ (MLE plug-in)} \\
		\cmidrule(lr){4-5}\cmidrule(lr){6-7}
		Distribution & $n$ & $I_5^{\mathrm{true}}$ & Bias & RMSE & Bias & RMSE \\
		\hline
		Gamma$(0.5,1)$ & 10 & 0.538 & $+0.0048$ & 0.1075 & $-0.0272$ & 0.0752 \\
		& 20 & 0.538 & $+0.0008$ & 0.0733 & $-0.0141$ & 0.0508 \\
		& 30 & 0.538 & $-0.0014$ & 0.0591 & $-0.0093$ & 0.0412 \\
		& 50 & 0.538 & $+0.0011$ & 0.0456 & $-0.0056$ & 0.0303 \\[0.5ex]
		Gamma$(2,1)$ & 10 & 0.311 & $+0.0007$ & 0.0713 & $-0.0219$ & 0.0628 \\
		& 20 & 0.311 & $+0.0005$ & 0.0476 & $-0.0102$ & 0.0420 \\
		& 30 & 0.311 & $+0.0003$ & 0.0383 & $-0.0070$ & 0.0344 \\
		& 50 & 0.311 & $-0.0003$ & 0.0293 & $-0.0040$ & 0.0261 \\[0.5ex]
		Gamma$(5,1)$ & 10 & 0.203 & $+0.0001$ & 0.0475 & $-0.0152$ & 0.0447 \\
		& 20 & 0.203 & $+0.0005$ & 0.0328 & $-0.0071$ & 0.0309 \\
		& 30 & 0.203 & $-0.0003$ & 0.0254 & $-0.0051$ & 0.0242 \\
		& 50 & 0.203 & $-0.0003$ & 0.0197 & $-0.0033$ & 0.0187 \\[0.5ex]
		Gamma$(10,1)$ & 10 & 0.145 & $+0.0003$ & 0.0339 & $-0.0108$ & 0.0323 \\
		& 20 & 0.145 & $-0.0009$ & 0.0236 & $-0.0065$ & 0.0231 \\
		& 30 & 0.145 & $+0.0001$ & 0.0187 & $-0.0036$ & 0.0182 \\
		& 50 & 0.145 & $-0.0003$ & 0.0146 & $-0.0025$ & 0.0142 \\[0.5ex]
		Lognormal$(0,0.5)$ & 10 & 0.230 & $-0.0051$ & 0.0576 & $-0.0195$ & 0.0532 \\
		& 20 & 0.230 & $-0.0030$ & 0.0400 & $-0.0097$ & 0.0367 \\
		& 30 & 0.230 & $-0.0018$ & 0.0328 & $-0.0063$ & 0.0291 \\
		& 50 & 0.230 & $+0.0004$ & 0.0254 & $-0.0025$ & 0.0224 \\[0.5ex]
		Lognormal$(0,1)$ & 10 & 0.447 & $-0.0272$ & 0.1127 & $-0.0379$ & 0.0971 \\
		& 20 & 0.447 & $-0.0156$ & 0.0840 & $-0.0194$ & 0.0663 \\
		& 30 & 0.447 & $-0.0088$ & 0.0700 & $-0.0099$ & 0.0540 \\
		& 50 & 0.447 & $-0.0074$ & 0.0563 & $-0.0063$ & 0.0409 \\[0.5ex]
		Weibull$(0.8,1)$ & 10 & 0.489 & $-0.0029$ & 0.1004 & $-0.0309$ & 0.0915 \\
		& 20 & 0.489 & $-0.0013$ & 0.0705 & $-0.0146$ & 0.0625 \\
		& 30 & 0.489 & $-0.0010$ & 0.0580 & $-0.0102$ & 0.0514 \\
		& 50 & 0.489 & $-0.0016$ & 0.0436 & $-0.0063$ & 0.0380 \\[0.5ex]
		Weibull$(1.6,1)$ & 10 & 0.289 & $+0.0036$ & 0.0639 & $-0.0199$ & 0.0621 \\
		& 20 & 0.289 & $+0.0012$ & 0.0441 & $-0.0104$ & 0.0431 \\
		& 30 & 0.289 & $+0.0019$ & 0.0355 & $-0.0060$ & 0.0346 \\
		& 50 & 0.289 & $+0.0004$ & 0.0261 & $-0.0044$ & 0.0256 \\[0.5ex]
		Lomax$(3,1)$ & 10 & 0.518 & $-0.0262$ & 0.1224 & $+0.0038$ & 0.0445 \\
		& 20 & 0.518 & $-0.0218$ & 0.0913 & $+0.0022$ & 0.0300 \\
		& 30 & 0.518 & $-0.0136$ & 0.0764 & $+0.0008$ & 0.0239 \\
		& 50 & 0.518 & $-0.0095$ & 0.0631 & $+0.0004$ & 0.0184 \\[0.5ex]
		Lomax$(5,1)$ & 10 & 0.472 & $-0.0129$ & 0.1036 & $+0.0014$ & 0.0209 \\
		& 20 & 0.472 & $-0.0097$ & 0.0771 & $+0.0006$ & 0.0146 \\
		& 30 & 0.472 & $-0.0067$ & 0.0648 & $+0.0005$ & 0.0115 \\
		& 50 & 0.472 & $-0.0022$ & 0.0478 & $+0.0000$ & 0.0092 \\
		\hline
	\end{tabular}
\end{table}

\begin{table}[!ht]
\centering
\small
\caption{Validation of the bias decomposition (Corollary~\ref{def-bias}) via Monte Carlo estimation of the $\Delta_{n,r}$ components ($R=300{,}000$ replications).}
\label{tab:delta_validation}
\begin{tabular}{lrrrrrrr}
\hline
Distribution & $n$ & $\Delta_{n,1}$ & $\Delta_{n,2}$ & $\Delta_{n,3}$ & $\Delta_{n,4}$ & Theory bias & MC bias \\
\hline
Lognormal$(0,1)$ & 10 & $+0.0010$ & $-0.0231$ & $-0.0475$ & $-0.0737$ & $-0.0259$ & $-0.0272$ \\
 & 20 & $+0.0010$ & $-0.0148$ & $-0.0277$ & $-0.0434$ & $-0.0188$ & $-0.0156$ \\
 & 30 & $+0.0025$ & $-0.0042$ & $-0.0153$ & $-0.0281$ & $-0.0085$ & $-0.0088$ \\
 & 50 & $+0.0017$ & $-0.0050$ & $-0.0106$ & $-0.0165$ & $-0.0070$ & $-0.0074$ \\[0.5ex]
Lomax$(3,1)$ & 10 & $+0.0026$ & $-0.0250$ & $-0.0561$ & $-0.0907$ & $-0.0312$ & $-0.0262$ \\
 & 20 & $+0.0002$ & $-0.0170$ & $-0.0341$ & $-0.0565$ & $-0.0224$ & $-0.0218$ \\
 & 30 & $+0.0016$ & $-0.0108$ & $-0.0270$ & $-0.0435$ & $-0.0127$ & $-0.0136$ \\
 & 50 & $-0.0026$ & $-0.0128$ & $-0.0237$ & $-0.0343$ & $-0.0099$ & $-0.0095$ \\
\hline
\end{tabular}
\end{table}
}

}

%\clearpage

{
\section{Empirical application}\label{empirical application}

In order to illustrate the practical use of the proposed estimator $\widehat{I}_5$, we analyze the distribution of GDP per capita across the Americas. The dataset comprises $n=34$ countries and territories in the Americas region in 2023. The raw data, measured in international dollars at 2021 PPP prices, were obtained from the Our World in Data repository \citep{OWID2024} and converted into units of USD$\times 10^{3}$. The final sample spans the full income spectrum of the region, ranging from low-income economies to high-income countries.

Table~\ref{tab:gdp_americas} lists the 34 countries together with their GDP per capita values sorted in ascending order. The dataset exhibits strong positive skewness and high variability, with GDP per capita ranging from USD\,2.956 thousand (Haiti) to USD\,74.578 thousand (United States). The coefficient of variation equals $0.623$, confirming substantial income dispersion across the region.

}
\begin{table}[!ht]
\centering
\caption{GDP per capita (USD $\times 10^{3}$, PPP constant 2021 prices) for 34 countries and territories in the Americas, 2023.}
\label{tab:gdp_americas}
\begin{tabular}{lr@{\qquad}lr}
\hline
Country & GDP & Country & GDP \\
\hline
Haiti                              &  2.956 & Barbados                    & 19.224 \\
Honduras                           &  6.468 & Mexico                      & 22.143 \\
Nicaragua                          &  7.487 & Dominican Republic          & 23.088 \\
Bolivia                            &  9.844 & Saint Lucia                 & 23.403 \\
Jamaica                            & 10.291 & Costa Rica                  & 26.293 \\
El Salvador                        & 11.404 & Argentina                   & 27.105 \\
Guatemala                          & 12.389 & Antigua and Barbuda         & 28.967 \\
Belize                             & 12.455 & Chile                       & 29.463 \\
Ecuador                            & 14.472 & Saint Kitts and Nevis       & 30.409 \\
Peru                               & 15.294 & Uruguay                     & 31.019 \\
Paraguay                           & 15.783 & Trinidad and Tobago         & 31.706 \\
Grenada                            & 16.946 & Bahamas                     & 33.106 \\
Dominica                           & 17.420 & Panama                      & 35.864 \\
St.\ Vincent \& the Grenadines     & 18.335 & Puerto Rico                 & 42.995 \\
Colombia                           & 18.692 & Guyana                      & 49.315 \\
Brazil                             & 19.018 & Canada                      & 55.919 \\
Suriname                           & 19.044 & United States               & 74.578 \\
\hline
\end{tabular}
\end{table}
{
We compute two inequality indices. First, the classical Gini coefficient $\widehat{G}=\widehat{I}_2$ using the unbiased L-statistic estimator \citep{Deltas2003}. Second, the proposed 5th Gini index estimator $\widehat{I}_5$ (with weights $(a_1,\ldots,a_5)=(-1,0,0,0,1)$), computed exactly via the L-statistic representation \eqref{Ihat-L-exact}. To quantify sampling uncertainty, $95\%$ bootstrap confidence intervals are obtained from $B=5{,}000$ resamples with replacement from the observed data. Table~\ref{tab:app_results} summarizes the results. Both estimators confirm a moderate-to-high level of income inequality across the Americas. The classical Gini coefficient $\widehat{G}=0.3286$ captures the overall level of pairwise dispersion in the distribution, while the 5th Gini index $\widehat{I}_5=0.2773$ measures dispersion between the expected minimum and maximum of a random quintet of countries, normalized by five times the mean. The relationship $\widehat{I}_5 < \widehat{G}$ is consistent with the theoretical ordering: $I_5$ only contrasts the two extreme order statistics in a subsample of size $5$, whereas the Gini coefficient integrates dispersion across all pairs.
}

\clearpage

\begin{table}[!ht]
\centering
\caption{Summary statistics and inequality estimates for GDP per capita in the Americas, 2023 ($n=34$). Bootstrap $95\%$ confidence intervals are based on $B=5{,}000$ resamples.}
\label{tab:app_results}
\begin{tabular}{lrl}
\hline
Statistic & Value & \\
\hline
\multicolumn{3}{l}{\textit{Descriptive statistics (USD $\times 10^3$)}} \\
\quad Sample size ($n$) & 34 & \\
\quad Mean ($\overline{X}$) & 23.909 & \\
\quad Median & 19.134 & \\
\quad Std.\ deviation & 14.897 & \\
\quad Minimum & 2.956 & (Haiti) \\
\quad Maximum & 74.578 & (United States) \\
\quad Coefficient of variation & 0.623 & \\[0.5ex]
\multicolumn{3}{l}{\textit{Inequality measures}} \\
\quad Gini ($\widehat{G}$) & 0.3286 & [0.2416,\; 0.3979] \\
\quad 5th Gini ($\widehat{I}_5$) & 0.2773 & [0.2000,\; 0.3332] \\
\hline
\end{tabular}
\end{table}

\section{Concluding remarks}\label{concluding remarks}
{
In this paper, we provide a unified bias analysis for a broad class of rank-based inequality measures defined through linear combinations of expected order statistics, encompassing the classical Gini coefficient, the $m$th Gini index, the extended $m$th Gini index, and particular cases of the $S$-Gini index (see Remark~\ref{rem-1}). The proposed framework connects naturally to spectral inequality measures and places the entire family within the axiomatically characterized class of Lorenz-consistent measures. We derive a general finite-sample bias decomposition for the canonical U-statistic-type estimator $\widehat{I}_m$ in terms of the factors $\Delta_{n,r}$, obtain an explicit Laplace-transform characterization of each $\Delta_{n,r}$ (Proposition~\ref{prop-main}) that serves as both a simulation-free computational tool and an analytical source of intuition for the asymptotic result, and establish asymptotic unbiasedness under mild moment and lower-tail conditions. The main contribution is the proof of exact unbiasedness under gamma populations, extending several existing results for Gini-type estimators \citep{Deltas2003,Baydil2025,Vila2025,Vila2026} via the Dirichlet property of normalized gamma samples. The Monte Carlo study confirms these findings and reveals, by comparison with a parametric MLE plug-in, that the correctly specified plug-in incurs systematic negative bias even under gamma populations, whereas $\widehat{I}_5$ remains exactly unbiased; under non-gamma populations the plug-in is more efficient but $\widehat{I}_5$ is distribution-free and robust to misspecification. An empirical illustration analyzes income inequality in GDP per capita across $34$ countries in the Americas. Extensions to other distributional families and the development of bias-corrected estimators for non-gamma populations are currently under investigation.}

%\clearpage
%\clearpage

\paragraph*{Acknowledgements}
The research was supported in part by CNPq and CAPES grants from the Brazilian government.

\paragraph*{Disclosure statement}
There are no conflicts of interest to disclose.

%%%%%%%%%%%%%%%%%%%%%%%%%%%%%%%%%%%%%%%%%%%%%%%%%%%%%%%%%%%%%

%%\bibliographystyle{apalike}
%%\bibliography{references}

%\begin{appendices}
%	\section{Proofs of the results in Section \ref{sec:02}}\label{Appendix A}
%
%\end{appendices}

%\bibliography{sn-bibliography}% common bib file
%%% if required, the content of .bbl file can be included here once bbl is generated
%%%\input sn-article.bbl

\end{document}